\documentclass[11pt]{article}

\usepackage[utf8]{inputenc} 
\usepackage{amsmath}
\usepackage{graphicx}
\usepackage{subcaption}
\usepackage[top=30mm, bottom=30mm, left=27mm, right=27mm]{geometry}
\usepackage{amsfonts}
\usepackage{mathtools}
\usepackage{subcaption}
\usepackage{amssymb}
\usepackage{amsthm}
\usepackage{chngpage}
\usepackage{hyperref}
\hypersetup{
	colorlinks=true,
	linkcolor=black,
	citecolor=black,
	filecolor=black,      
	urlcolor=black,
}
\usepackage{enumerate}
\usepackage{makeidx}
\usepackage{bm}
\usepackage{thmtools}
\usepackage{thm-restate}

\makeindex

\makeatletter
\renewenvironment{proof}[1][\proofname] {\par\pushQED{\qed}\normalfont\topsep6\p@\@plus6\p@\relax\trivlist\item[\hskip\labelsep\bfseries#1\@addpunct{.}]\ignorespaces}{\popQED\endtrivlist\@endpefalse}
\makeatother

\newcommand{\ex}{\operatorname{ex}}
\newcommand{\RR}{\mathbb{R}}
\newcommand{\PP}{\mathbb{P}}
\newcommand{\ZZ}{\mathbb{Z}}
\newcommand{\thth}{\textsuperscript{th} }
\newcommand*\diff{\mathop{}\!\mathrm{d}}

\newtheorem{proposition}{Proposition}[section]
\newtheorem{lemma}[proposition]{Lemma}
\newtheorem{theorem}[proposition]{Theorem}

\newtheorem{question}[proposition]{Question}
\usepackage{amsthm}

\theoremstyle{definition}

\newtheorem*{remark*}{Remark}
\newtheorem*{theorem*}{Theorem}
\newtheorem*{claim*}{Claim}

\def\so{\mathrm{SO}}

\title{Generalizations of the Ruzsa--Szemerédi and rainbow Turán problems for cliques
}
\author{W. T. Gowers \and Barnabás Janzer
}
\date{\vspace{-21pt}}

\begin{document}
	\maketitle
	
\begin{abstract}
Considering a natural generalization of the Ruzsa--Szemerédi problem, we prove that for any fixed positive integers $r,s$ with $r<s$, there are graphs on $n$ vertices containing $n^{r}e^{-O(\sqrt{\log{n}})}=n^{r-o(1)}$ copies of $K_s$ such that any $K_r$ is contained in at most one $K_s$. We also give bounds for the generalized rainbow Turán problem $\ex(n,H,\textnormal{rainbow-}F)$ when $F$ is complete. In particular, we answer a question of Gerbner, Mészáros, Methuku and Palmer, showing that there are properly edge-coloured graphs on $n$ vertices with $n^{r-1-o(1)}$ copies of $K_r$ such that no $K_r$ is rainbow.
\end{abstract}\vspace{-5pt}	

\section{Introduction}\label{sec.intrKrKsrainbow}

The famous Ruzsa--Szemerédi or $(6,3)$-problem 
is to determine how many edges there can be in a 3-uniform hypergraph on $n$ vertices if no six vertices span three or more edges. This rather specific-sounding problem turns out to have several equivalent formulations and bounds in both directions have had many applications. 
It is not difficult to prove an upper bound of $O(n^2)$: one first observes that if two edges have two vertices in common, then neither of them can intersect any other edges, and after removing all such pairs of edges one is left with a linear hypergraph, for which the bound is trivial. Brown, Erdős and Sós \cite{sos1973existence} gave a construction achieving $\Omega(n^{3/2})$ edges and asked whether the maximum is $o(n^2)$. 

The argument sketched in the previous paragraph shows that this question is equivalent to asking whether a graph on $n$ vertices such that no edge is contained in more than one triangle 
must contain $o(n^2)$ triangles. A positive answer to this question was given by Ruzsa and Szemerédi \cite{ruzsa1978triple}, who obtained a bound of $O(n^2/\log_*n)$ with the help of Szemerédi's regularity lemma. They also gave a construction showing that the number of triangles can be as large as $n^2e^{-O(\sqrt{\log n})}=n^{2-o(1)}$, so the exponent in their upper bound cannot be improved. 

One of the applications they gave of their upper bound was an alternative proof of Roth's theorem. Indeed, let $A$ be a subset of $\{1,\dots,N\}$ that contains no arithmetic progression of length $3$. Define a tripartite graph $G$ with vertex classes $X=\{1,2,\dots,N\}$, $Y=\{1,2,\dots,2N\}$ and $Z=\{1,2,\dots,3N\}$, where if $x\in X$, $y\in Y$ and $z\in Z$, then $xy$ is an edge if and only if $y-x\in A$, $yz$ is an edge if and only if $z-y\in A$ and $xz$ is an edge if and only if $(z-x)/2\in A$. Note that these are the edges of the triangles with vertices belonging to triples of the form $(x,x+a,x+2a)$ with $x\in X$ and $a\in A$. If $xyz$ is a triangle in this graph, then $a=y-x, b=z-y, c=(z-x)/2$ satisfy $a,b,c\in A$ and $a+b=2c$, which gives us an arithmetic progression of length $3$ in $A$ unless $y-x=z-y$. Thus, the only triangles are the `degenerate' ones of the form $(x,x+a,x+2a)$, which implies that each edge is contained in at most one triangle. Therefore, the number of triangles is $o(n^2)$ (where $n=6N$). We also have that for each $a\in A$ there are $N$ triangles of the form $(x,x+a,x+2a)$, so $|A|=o(N)$. 

As Ruzsa and Szemerédi also observed, this argument can be turned round: it tells us that if $A$ has density $\alpha$, then there is a graph with $6N$ vertices and $\alpha N^2$ triangles such that each edge is contained in at most one triangle. Since Behrend proved \cite{behrend1946sets} that there exists a subset $A$ of $\{1,\dots,N\}$ of size $Ne^{-O(\sqrt{\log{N}})}$ that does not contain an arithmetic progression of length $3$, this gives the lower bound mentioned above.

Several related questions have been studied, as well as applications and generalizations of the Ruzsa--Szemerédi problem: see for example \cite{alon2006extremal,alon2012nearly}. A natural generalization that we believe has not been considered is the following generalized Turán problem.

\begin{question}\label{question_RuzsaSzemeredigen}
	Let $r$ and $s$ be positive integers with $1\leq r<s$. Let $G$ be a graph on $n$ vertices such that any of its subgraphs isomorphic to $K_r$ is contained in at most one subgraph isomorphic to $K_s$. What is the largest number of copies of $K_s$ that $G$ can contain?
\end{question}

The Ruzsa--Szemerédi problem is the case $r=2, s=3$ of Question \ref{question_RuzsaSzemeredigen}, and the answer is trivially $\Theta(n)$ if $r=1$. One can easily deduce from the graph removal lemma an upper bound of $o(n^r)$ when $r\geq 2$. In the case $r=2$, the construction for the lower bound can be generalized (for example, by using $h$-sum-free sets from \cite{alon2006characterization}) to get a lower bound of $n^2e^{-O(\sqrt{\log{n}})}$. However, there is no obvious way of generalizing the algebraic construction for $r\geq 3$. We shall present a geometric construction instead, in order to prove the following result, which is the first of the two main results of this paper.

\begin{theorem}\label{theorem_KrKs}
	For each $1\leq r<s$ and positive integer $n$ there is a graph on $n$ vertices with $n^{r}e^{-O(\sqrt{\log{n}})}=n^{r-o(1)}$ copies of $K_s$ such that every $K_r$ is contained in at most one $K_s$.
\end{theorem}\medskip

We shall also use a modification of our construction to answer a question about rainbow colourings. Given an edge-colouring of a graph $G$, we say that a subgraph $H$ is \textit{rainbow} if all of its edges have different colours. We denote by $\ex^*(n,H)$ the maximal number of edges that a graph on $n$ vertices can contain if it can be properly edge-coloured (that is, no two edges of the same colour meet at a vertex) in such a way that it contains no rainbow copy of $H$. The rainbow Turán problem (i.e., the problem of estimating $\ex^*(n,H)$) was introduced by Keevash, Mubayi, Sudakov and Verstra\"ete \cite{keevash2007rainbow}, and was studied for several different families of graphs $H$, such as complete bipartite graphs \cite{keevash2007rainbow}, even cycles \cite{keevash2007rainbow,das2013rainbow} and paths \cite{johnston2016rainbow,ergemlidze2019rainbow}. Gerbner, Mészáros, Methuku and Palmer \cite{gerbner2019generalized} considered the following generalized rainbow Turán problem (analogous to a generalization of the usual Turán problem introduced by Alon and Shikhelman \cite{alon2016many}). Given two graphs $H$ and $F$, let $\ex(n,H,\textnormal{rainbow-}F)$ denote the maximal number of copies of $H$ that a properly edge-coloured graph on $n$ vertices can contain if it has no rainbow copy of $F$. Note that $\ex^*(n,H)$ is the special case $\ex(n,K_2,\textnormal{rainbow-}H)$. The authors of \cite{gerbner2019generalized} focused on the case $H=F$ and obtained several results, for example when $H$ is a path, cycle or a tree, and also gave some general bounds. One of their concluding questions was the following.

\begin{question}[Gerbner, Mészáros, Methuku and Palmer \cite{gerbner2019generalized}]\label{question_rbKr}
	What is the order of magnitude of $\ex(n,K_r,\textnormal{rainbow-}K_r)$ for $r\geq 4$?
\end{question}

{For fixed $r$,} a straightforward double-counting argument shows that if $H$ has $r$ vertices, then $\ex(n,H,\textnormal{rainbow-}H)=O(n^{r-1})$. 
{Indeed, if  $G$ is a graph with $n$ vertices that contains no rainbow copy of $H$, then every copy of $H$ contains two edges of the same colour. But the number of such pairs of edges is at most $\binom n2\frac{n-2}2=O(n^3)$, since there are at most $\frac{n-2}2$ edges with the same colour as any given edge, and each such pair can be extended to at most $r!n^{r-4}$ copies of $H$.}

The authors above improved this bound to $o(n^{r-1})$, and gave an example that shows that $\ex(n,K_r,\textnormal{rainbow-}K_r)=\Omega(n^{r-2})$. They also asked whether there is a graph $H$ for which the exponent $r-1$ in the upper bound is sharp. Our next result shows that $H=K_r$ is such a graph.

\begin{theorem}\label{theorem_rainbow}
	For each $r\geq 4$ we have $\ex(n,K_r,\textnormal{rainbow-}K_r)=n^{r-1-o(1)}$.
\end{theorem}
Note that a triangle is always rainbow in a proper edge-colouring, so we trivially have $\ex(n,K_r,\textnormal{rainbow-}K_r)=0$ for $r<4$.\medskip

In fact, our method can be used to prove the following more general result.

\begin{restatable}{theorem}{rainbowgeneral}\label{theorem_rbgeneral}
	Let $r\geq 4$, let $H$ be a graph, 
	and let $H$ have a proper edge-colouring with no rainbow $K_r$. Suppose that for each vertex $v$ of $H$ there is a $p_v\in\RR^m$, and for each colour $\kappa$ in the colouring there is a non-zero vector $z_\kappa$ such that for every edge $vw$ of colour $\kappa$, $z_\kappa$ is a linear combination of $p_v$ and $p_w$ with non-zero coefficients. Then $\ex(n,H,\textnormal{rainbow-}K_r)\geq n^{m_0-o(1)}$, where $m_0$ is the dimension of the subspace of $\RR^m$ spanned by the points $p_v$. 
\end{restatable}


It is easy to see that Theorem \ref{theorem_rainbow} is a special case of Theorem \ref{theorem_rbgeneral}, but Theorem \ref{theorem_rbgeneral} also allows us to determine the behaviour of $\ex(n,H,\textnormal{rainbow-}K_r)$ for several other natural choices of $H$. We give some examples in Section \ref{sec_examples}.

Theorem \ref{theorem_rbgeneral} is `almost equivalent' to the following, slightly weakened, alternative version.
\bigskip

\noindent
\textbf{Theorem \ref{theorem_rbgeneral}$'$.}
\textit{Let $r\geq 4$, let $H$ be a graph, and let $c$ be a proper edge-colouring of $H$ without a rainbow $K_r$. Suppose that for each vertex $v\in V(H)$ we have a vector $p_v\in\RR^{m-1}$, and for each colour $\kappa$ of $c$ the lines through the pairs $p_v, p_w$ with $c(vw)=\kappa$ are either all parallel, or all go through the same point and that point is different from $p_v,p_w$ unless $p_v=p_w$. Assume that no $(m-2)$-dimensional affine subspace contains all the points $p_v$. Then $\ex(n,H,\textnormal{rainbow-}K_r)\geq n^{m-o(1)}$.}\bigskip


It is easy to see that Theorem \ref{theorem_rbgeneral}$'$ is equivalent to the weakened version of Theorem \ref{theorem_rbgeneral} where we make the additional assumption that each $p_v$ is non-zero. Indeed, given a configuration of points $p_v$ as in Theorem \ref{theorem_rbgeneral} (with $m=m_0$), we can project it from the origin to an appropriate affine $(m-1)$-dimensional subspace not going through the origin to get a configuration as in Theorem \ref{theorem_rbgeneral}$'$. Conversely, a configuration of points $p_v$ as in Theorem \ref{theorem_rbgeneral}$'$ gives a configuration as in Theorem \ref{theorem_rbgeneral} by taking the points $p_v\times \{1\}\in \RR^m$.

\section{The idea of the construction, and a preliminary lemma}\label{subsec_constr}
We now briefly describe the construction used in our proof of Theorem \ref{theorem_KrKs}. For simplicity, we focus on the case $r=2, s=3$, i.e., the Ruzsa--Szemerédi problem.

Consider the $d$-dimensional sphere $S^d=\{x\in\RR^{d+1}:\Vert x\Vert=1\}$. (We will choose $d$ to be about $\sqrt{\log{n}}$.) Join two points of the sphere by an edge if the angle between the corresponding vectors is between $2\pi/3-\delta$ and $2\pi/3+\delta$, where $\delta$ is some appropriately chosen small number (roughly $e^{-\sqrt{\log{n}}}$). Then there are `few' triangles containing any given edge, since if $xy$ is an edge then any point $z$ such that $xyz$ is a triangle is restricted to lie in a small neighbourhood around the point $-(x+y)$. However, there are `many' edges, since the edge-neighbourhood of a point is a set of points around a codimension-1 surface, which is much larger then the neighbourhood of a single point. Choosing the constants appropriately, we can achieve that if we pick $n$ random points then any two of them form an edge with probability $n^{-o(1)}$, and any three of them form a triangle with probability $n^{-1-o(1)}$. Then any edge is expected to be in $n^{-o(1)}$ triangles and there are $n^{2-o(1)}$ edges. After some modification, we get a graph with $n^{2-o(1)}$ triangles in which any edge extends to at most one triangle.

The general construction is quite similar. We want to define the edges in such a way that knowing the position of any $r$ of the vertices of a $K_s$ restricts the remaining $s-r$ vertices to small neighbourhoods around certain points, but knowing the position of $i$ points with $i<r$ only restricts the remaining points to a neighbourhood of a codimension-$i$ surface. For example, when $(r,s)=(3,4)$, we can define our graph by joining two points if the angle between the corresponding vectors is close to the angle given by two vertices of a regular tetrahedron (centred at the origin).

In fact, our construction and the construction of Ruzsa and Szemer\'edi based on the Behrend set are more similar than they might at first appear, which also explains why they give similar bounds (namely $n^2e^{-O(\sqrt{\log{n}})}$ for the case $r=2,s=3$). Behrend's construction \cite{behrend1946sets} of a large set with no arithmetic progression of length 3 starts by observing that for any positive integers $k,d$ there is some $m$ such that the grid $\{1,\dots,k\}^d$ intersects the sphere $\{x\in\RR^d:\Vert x \Vert^2=m\}$ in a set $A$ consisting of at least $k^d/(dk^2)$ points. This set $A$ has no arithmetic progression of length $3$. (In Behrend's construction, this is transformed into a subset of $\ZZ$ using an appropriate map, but this is unnecessary for our purposes.) Repeating the construction from Section \ref{sec.intrKrKsrainbow}, we define a tripartite graph $G$ on vertex set $X\cup Y\cup Z$ where $X=\{1,\dots,k\}^d, Y=\{1,\dots,2k\}^d, Z=\{1,\dots,3k\}^d$, and edges given by the edges of the triangles $(x,x+a,x+2a)\in X\times Y\times Z$ for $x\in X, a\in A$. Explicitly, for $x\in X, y\in Y, z\in Z$, we join $x$ and $y$ if $\Vert x-y\Vert =m^{1/2}$ (and $y_i\geq x_i$ for all $i$), we join $y$ and $z$ if $\Vert z-y\Vert =m^{1/2}$ (and $z_i\geq y_i$), and we join $x$ and $z$ if $\Vert x-z\Vert =2m^{1/2}$ (and $z_i\geq x_i$). This gives the same phenomenon as our construction: the neighbourhood of a point $x$ is given by a codimension-1 condition, but the joint neighbourhood of two points is a single point, since $y$ must be the midpoint of $x$ and $z$.

\medskip

We conclude this section with the following technical fact, whose proof we include for completeness. Given unit vectors $v,w$, we write $\angle(v,w)$ for the angle between $v$ and $w$ -- that is, for $\cos^{-1}(\langle v,w\rangle)$.
\begin{lemma}\label{lemma_spherebounds}
	There exist constants $0<\alpha<B$ such that the following holds. Let $d$ be a positive integer, let $0<\rho\leq 2$ and let $v\in S^d$. Let $X_\rho=\{w\in S^d: \Vert v-w\Vert <\rho\}$. Let $\mu$ denote the usual probability measure on $S^d$. Then
	\[\alpha^d\rho^d\leq \mu(X_\rho)\leq B^d\rho^d.\]
Furthermore, for any $-1<\xi<1$ there exists $\beta>0$ such that for every positive integer $d$, every point $v\in S^d$, and every $0\leq\delta\leq 2$, the set $Y_{\xi,\delta}=\{w\in S^d: |\langle v,w\rangle-\xi|<\delta\}$ has
\[\mu(Y_{\xi,\delta})\geq \beta^d\delta.\]
\end{lemma}

\begin{proof}
	Using the usual spherical coordinate system, we see that for $0\leq \varphi\leq \pi$ the set $Z_\varphi=\{w\in S^d: \angle(v,w)<\varphi\}$ satisfies
	\begin{equation}\label{eq_spherecalc}
	\mu(Z_\varphi)=\frac{\int_{0}^{\varphi} \sin^{d-1}{\theta}\diff\theta}{\int_{0}^{\pi} \sin^{d-1}{\theta}\diff\theta}.
	\end{equation}

	But we have $\theta\geq \sin\theta\geq \frac{2}{\pi}\theta$ for $0\leq \theta\leq \pi/2$. Thus, ${\int_{0}^{t} \sin^{d-1}{\theta}\diff\theta}$ is between $\frac{c_1^{d-1}}{d}t^d$ and $\frac{1}{d}t^d$ for all $0\leq t\leq \pi /2$ (for some constant $0<c_1<1$). Using this bound for both the numerator and the denominator in \eqref{eq_spherecalc}, we deduce that $\alpha_0^d\varphi^{d}\leq \mu(Z_\varphi)\leq B_0^d\varphi^d$ for some absolute constants $0<\alpha_0<B_0$. But if $\angle(v,w)=\varphi$ and $\Vert v-w\Vert =\rho\leq 2$, then $\rho\leq \varphi\leq \rho\pi/2$, so $Z_{\rho} \subseteq X_\rho\subseteq Z_{\rho\pi/2}$. The first claim follows.
	
	For the second claim, let $0<\varphi<\pi$ such that $\xi=\cos\varphi$ and let $\epsilon_0=\min\{\varphi/2,(\pi-\varphi)/2\}$. Write $W_{\varphi,\epsilon}=\{w\in S^d: \varphi-\epsilon<\angle(v,w)<\varphi+\epsilon\}$. For $0<\epsilon<\epsilon_0$ we have 
	\[	\mu(W_{\varphi,\epsilon})=\frac{\int_{\varphi-\epsilon}^{\varphi+\epsilon} \sin^{d-1}{\theta}\diff\theta}{\int_{0}^{\pi/2} \sin^{d-1}{\theta}\diff\theta}.\]
	
	But also $\sin\theta\geq \min\{\sin(\varphi-\epsilon_0),\sin(\varphi+\epsilon_0)\}=\sin(\epsilon_0)$ when $\varphi-\epsilon_0\leq \theta\leq \varphi+\epsilon_0$. Writing $\beta_0=\sin(\epsilon_0)>0$, it follows that whenever $\epsilon<\epsilon_0$, then $\mu(W_{\varphi,\epsilon})\geq\frac{2\epsilon \beta_0^{d-1}}{\pi/2}$.
	However, we have $|\cos(\theta)-\cos(\varphi)|\leq |\theta-\varphi| 
	$, so $Y_{\xi,\delta}\supseteq W_{\varphi,\delta}$. 
	Choosing some sufficiently small $\beta$, the second claim follows.
\end{proof}

\section{The generalized Ruzsa--Szemerédi problem}

In this section we prove the first of our main results, Theorem \ref{theorem_KrKs}. In the case $r=2, s=3$, the construction is based, as we saw in Section~\ref{subsec_constr}, on the observation that if we wish to find three vectors in $S^d=\{x\in \RR^{d+1}: \Vert x\Vert =1\}$ in such a way that the angle between any two of them is $120^\circ$, and if we choose the vertices one by one, then there are $d$ degrees of freedom for the first vertex and $d-1$ for the second, but the third is then uniquely determined. This gives us an example of a `continuous graph' with `many' edges, such that each edge is in exactly one triangle, and a suitable perturbation and discretization of this graph gives us a finite graph with $n^{2-o(1)}$ 
triangles such that each edge belongs to at most one triangle. 

To generalize this to arbitrary $(r,s)$ we need to find a configuration of $s$ unit vectors (where by `configuration' we mean an $s\times s$ symmetric matrix that specifies the angles, or equivalently inner products, between the unit vectors) with the property that if we choose the points of the configuration one by one, then for $i\leq r$ the $i$\thth point can be chosen with $d+1-i$ degrees of freedom, but from the $(r+1)$\textsuperscript{st} point onwards all points are uniquely determined. It turns out that all we have to do is choose an arbitrary collection of $s$ points $p_1,\dots,p_s$ in general position from the sphere $S^{r-1}$ and take the angles $\angle(p_i,p_j)$. To see that this works, suppose we we wish to choose $x_1,\dots,x_s\in S^d$ one by one in such a way that $\langle x_i,x_j\rangle=\langle p_i,p_j\rangle$ for every $i,j$. Suppose that we have chosen $x_1,\dots,x_r$ and let $V$ be the $r$-dimensional subspace that they generate. Let $u_{r+1}$ be the orthogonal projection of $x_{r+1}$ to $V$. Then $\langle u_{r+1},x_i\rangle=\langle x_{r+1},x_i\rangle$ for each $i\leq r$, and $u_{r+1}\in V$, so $u_{r+1}$ is uniquely determined. Furthermore, since the angles $\langle p_i,p_j\rangle$ are equal to the angles $\langle x_i,x_j\rangle$ when $i,j\leq r$ and to the angles $\langle x_i,u_{r+1}\rangle$ when $i\leq r, j=r+1$, and $p_{r+1}$ is a unit vector, it must be that $u_{r+1}$ is a unit vector, which implies that $x_{r+1}=u_{r+1}$. Since this argument made no use of the ordering of the vectors, it follows that any $r$ vectors in a configuration determine the rest, as claimed.

We shall now use this observation as a guide for constructing a finite graph with many copies of $K_s$ such that each $K_r$ is contained in at most one $K_s$.

As above, pick $s$ `reference' points $p_1, \dots, p_s$ in general position on the sphere $S^{r-1}$. Since for any set $B\subseteq \{1,\dots,s\}$ of size $r$ the points $p_b$ ($b\in B$) form a basis of $\RR^r$, we may write, for any $a$, 
\[p_a=\sum_{b\in B}{\lambda_{B,a,b}p_b}\]
for some real constants $\lambda_{B,a,b}$. \medskip

For any $c>0$ and positive integers $N,d$ we define an $s$-partite random graph $G_{N,d,c}$ as follows. (The graph will also depend on $r,s,p_1,\dots,p_s$, but for readability we drop these dependencies from the notation.) Consider the usual probability measure on the $d$-sphere $S^d$. Pick, independently and uniformly at random, $sN$ points $x_{a,i}$ ($1\leq a\leq s, 1\leq i\leq N$) on $S^d$: these points form the vertex set. Join two points $x_{a,i}$ and $x_{b,j}$ by an edge if $a\not =b$ and $|\langle x_{a,i},x_{b,j}\rangle-\langle p_a,p_b\rangle|<c$. Write $V_a=\{x_{a,i}: 1\leq i\leq N\}$ so that $G_{N,d,c}$ is $s$-partite with classes $V_1,\dots,V_s$.

We also define a graph $G_{N,d,c}'$ as follows. Let $M_0$ be the maximum among all values of $|\lambda_{B,a,b}|$ and $\lambda_{B,a,b}^2$, and let $M=2(r+1)\sqrt{M_0}$. 
 Then $G_{N,d,c}'$ is obtained from $G_{N,d,c}$ by deleting all vertices $x_{a,i}$ for which there is another vertex $x_{a,j}$ ($i\not =j$) such that $\Vert x_{a,i}-x_{a,j}\Vert <M\sqrt{c}$.

This graph is designed to be finite and to have the property that any copy of $K_s$ must be close to a configuration with angles determined by the points $p_1,\dots,p_s$. The vertex deletions are there to ensure that the vertices are reasonably well separated. This will imply that no $K_r$ is contained in more than one $K_s$, since once $r$ vertices of a $K_s$ are chosen, the remaining vertices are constrained to lie in small neighbourhoods.

\begin{lemma}
	The graph $G_{N,d,c}'$ has the property that any of its subgraphs isomorphic to $K_r$ is contained in at most one subgraph isomorphic to $K_s$ (for any choices of $r,s,p_1,\dots,p_s,N,d,c$).
\end{lemma}
\begin{proof}
	Let $x_{a_1,i_i},\dots, x_{a_r,i_r}$ be points that form a $K_r$. Then necessarily all $a_t$ are distinct. Suppose that we have two extensions $H_1,H_2$ of this $K_r$ to a $K_s$. Then both $H_1$ and $H_2$ intersect each class $V_a$ in exactly one point. We now show that for each $a$ this point must be the same for $H_1$ and $H_2$, which will imply the lemma.
	
	Suppose that $H_1$ intersects $V_a$ in point $x$. Write $B=\{a_1,\dots,a_r\}$. Then
	\begin{align*}
	\left\Vert  x-\sum_{t=1}^{r}\lambda_{B,a,a_t}x_{a_t}\right\Vert ^2&=\left\langle x-\sum_{t=1}^{r}\lambda_{B,a,a_t}x_{a_t}, x-\sum_{t=1}^{r}\lambda_{B,a,a_t}x_{a_t}\right\rangle\\
	&=\langle x,x\rangle-2\sum_{t=1}^r{\lambda_{B,a,a_t}\left\langle x, x_{a_t}\right\rangle}+\sum_{t,t'=1}^{r}{\lambda_{B,a,a_t}\lambda_{B,a,a_{t'}}\langle x_{a_t}, x_{a_{t'}}\rangle}\\
	&\leq 1 -2\sum_{t=1}^r\lambda_{B,a,a_t}\langle p_a,p_{a_t}\rangle+2c\sum_{t=1}^r{|\lambda_{B,a,a_t}|}\\
	&\hspace{2cm}+\sum_{t,t'=1}^{r}{\lambda_{B,a,a_t}\lambda_{B,a,a_{t'}}\langle p_{a_t},p_{a_{t'}}\rangle}+c\sum_{t,t'=1}^{r}{|\lambda_{B,a,a_t}\lambda_{B,a,a_{t'}}|}\\
	&\leq \left\langle p_a-\sum_{t=1}^{r}\lambda_{B,a,a_t}p_{a_t},p_a-\sum_{t=1}^{r}\lambda_{B,a,a_t}p_{a_t}\right\rangle + 2rcM_0+r^2cM_0\\
	&= (r^2+2r)cM_0.
	\end{align*}
	It follows that $\left\Vert  x-\sum_{t=1}^{r}\lambda_{B,a,a_t}x_t\right\Vert < (r+1)\sqrt{M_0c}$. Similarly, if $H_2$ intersects $V_a$ in point $y$ then $\left\Vert  y-\sum_{t=1}^{r}\lambda_{B,a,a_t}x_t\right\Vert < (r+1)\sqrt{M_0c}$. Hence $\Vert x-y\Vert < 2(r+1)\sqrt{M_0c}=M\sqrt c$. By the definition of $G'_{N,d,c}$, we must have $x=y$.
\end{proof}

To prove Theorem \ref{theorem_KrKs}, it suffices to show that the expected number of copies of $K_s$ in $G_{N,d,c}'$ is at least $N^{r}e^{-O(\sqrt{\log{N}})}$ for suitable choices of $d$ and $c$. For this purpose we shall use the following technical lemma. For later convenience (in Section \ref{sec_rainbow}), we state it in a slightly more general form than required here, to allow the possibility that $r=s$ and the possibility that $p_1,\dots,p_s$ are not in general position (but still span $\RR^r$).

\begin{restatable}{lemma}{appendixlemma}\label{lemma_probabilityKs}
	Let $1\leq r\leq s$ be positive integers and let $p_1, \dots, p_s$ be points on $S^{r-1}$ such that $p_1,\dots, p_r$ form a basis of $\RR^r$. Then there exist constants $\alpha>0$ and $h$ 
	such that for any $d\geq r$ and $0<c<1$ the probability that a set $\{x_{a}:1\leq a \leq s\}$ of random unit vectors (chosen independently and uniformly) on $S^d$ 
	satisfies $|\langle x_a,x_b\rangle-\langle p_a,p_b\rangle|<c$ for all $a,b$ is at least $\alpha^{d}c^{d(s-r)/2+h}$.
\end{restatable}

We may think of the conclusion of Lemma \ref{lemma_probabilityKs} as follows. The dominant (smallest) factor in the probability above is the factor $c^{d(s-r)/2}$. The probability should be close to this because if we imagine placing the $s$ points one by one and we have already picked $x_1,\dots,x_i$ joined to each other, then
\begin{itemize}
	\item if $i<r$, then $x_{i+1}$ is restricted to a neighbourhood of a codimension-$i$ surface, so with reasonably large probability (comparable to $c^{i}$)
	 it is connected to all previous vertices;
	\item if $i\geq r$, then the linear dependencies between the points restrict $x_{i+1}$ to be in a ball of radius about $c^{1/2}$ around a certain point, which has measure about $c^{d/2}$ (which is much smaller than $c^{r}$).
\end{itemize}
The proof of Lemma \ref{lemma_probabilityKs} is given in an appendix.

\begin{proof}[Proof of Theorem \ref{theorem_KrKs}]
	By Lemma \ref{lemma_spherebounds}, there are constants $c_0,B,C$ such that if $c<c_0$ then the probability that a given vertex $x_{a,i}$ is removed from $G_{N,d,c}$ when forming $G_{N,d,c}'$ is at most $NB^d(M\sqrt{c})^d\leq NC^dc^{d/2}$. Here $B>0$ is an absolute constant and the constants $C,c_0>0$ depend on $r,s,p_1,\dots,p_s$ only. Moreover, the event `$x_{a,i}$ is removed' is independent of any event of the form `$x_{1,i_1},\dots,x_{s,i_s}$ form a $K_s$ in $G_{N,d,c}$'. Using Lemma \ref{lemma_probabilityKs}, we deduce that the probability that $x_{1,i_i},\dots,x_{s,i_s}$ is contained in $G_{N,d,c}'$ and forms a $K_s$ is at least $(1-sNC^dc^{d/2})\alpha^dc^{d(s-r)/2+h}$ (where $\alpha, h$ depend on $r,s,p_1,\dots,p_s$ only). So the expected number of copies of $K_s$ in $G_{N,d,c}'$ is at least
	\[N^s(1-sNC^dc^{d/2})\alpha^dc^{d(s-r)/2+h}.\]
	
	If $c=(2sNC^d)^{-2/d}$, then this is at least
	\begin{equation}\label{eq_optimisation}
	\frac{1}{2}N^s\alpha^d\frac{1}{(2sC^d)^{s-r}}N^{r-s}(2sNC^d)^{-2h/d}\geq \eta^dN^{r-2h/d}=N^re^{-Ed-(2h/d)\log{N}}	
	\end{equation}
	for some constants $\eta>0$ and $E$ not depending on $N,d$.
	
Choosing $d=\lfloor \sqrt{\log{N}}\rfloor $, this is $N^re^{-O(\sqrt{\log{N}})}$, and $c<c_0$ when $N$ is sufficiently large. The result follows, as $G_{N,d,c}'$ has at most $Ns$ vertices.
\end{proof}

Note that our proof in fact also gives the correct (and trivial) lower bound $\Theta(n)$ in the case $r=1$, since if $r=1$ then $h=0$ so we may choose $d$ to be a constant and get $\Theta(N)$ in \eqref{eq_optimisation}.

\section{Generalized rainbow Turán numbers for complete graphs}\label{sec_rainbow}

We now turn to the proofs of  our results about generalized rainbow Turán numbers (Theorems \ref{theorem_rainbow} and \ref{theorem_rbgeneral}). First we recall a general result of Gerbner, Mészáros, Methuku and Palmer \cite{gerbner2019generalized}, which can be proved using the graph removal lemma.

\begin{proposition}[Gerbner, Mészáros, Methuku and Palmer \cite{gerbner2019generalized}]\label{proposition_rbHH}
	For any graph $H$ on $r$ vertices, we have $\ex(n,H,\textnormal{rainbow-}H)=o(n^{r-1})$.
\end{proposition}

In particular, we know that $\ex(n,K_r,\textnormal{rainbow-}K_r)=o(n^{r-1})$. We would like to match this with a lower bound of the form $n^{r-1-o(1)}$.

Before we prove such a bound, let us briefly discuss the ideas that underlie the proof. It is easy to show that a lower bound $\ex(n,K_4,\textnormal{rainbow-}K_4)\geq n^{3-o(1)}$ would imply that $\ex(n,K_r,\textnormal{rainbow-}K_r)\geq n^{r-1-o(1)}$ for all $r\geq 4$, so it suffices to consider the case $r=4$. However, when $r=4$ and $G$ is a properly edge-coloured graph with no rainbow $K_4$, then every triangle of $G$ is contained in at most three copies of $K_4$. (Indeed, if the vertices of the $K_3$ are $x,y,z$, then the only way that adding a further vertex $w$ can lead to a non-rainbow $K_4$ is if $wx$ has the same colour as $yz$, $wy$ has the same colour as $xz$ or $wz$ has the same colour as $xy$. But since the edge-colouring is proper, we cannot find more than one $w$ such that the same one of these three events occurs.) So it is natural to expect that our construction for Theorem \ref{theorem_KrKs} is relevant here.

To see how a similar construction gives the desired result, it is helpful, as earlier, to look at a simpler continuous example that serves as a guide to the construction. Consider the graph where the vertex set is $S^d$ and two unit vectors $v,w$ are joined if and only if $\langle v,w\rangle=-1/3$ (the angle between vectors that go through the origin and two distinct vertices of a regular tetrahedron). Then any $K_4$ in this graph must be given by the vertices of a regular tetrahedron. We colour an edge by the line that joins the origin to the midpoint of that edge. This is a proper colouring with the property that opposite edges have the same colour, so each $K_4$ is 3-coloured in this colouring. The construction we are about to describe is a suitable perturbation and discretization of this one.

For the discretized graph, we will again have `near-regular' tetrahedra forming $K_4$s. To ensure that each copy of $K_4$ is still rainbow, we shall have to modify the colouring slightly. We shall take only certain `allowed lines' as colours, and we shall colour an edge by the allowed line that is closest to the line through the midpoint (if that line is not very far -- otherwise we delete the edge). We need to choose the allowed lines in such a way that no two allowed lines are close (so that near-regular tetrahedra are still 3-coloured), but a large proportion of lines are close to an allowed line (so that not too many edges are deleted). This can be achieved using the following lemma.

\begin{lemma}\label{lemma_directions}
	There exists $\delta>0$ with the following property. For any $0< c_1< 1$ we can choose $L\geq (\delta/c_1)^d$ points $q_1,\dots,q_L$ on $S^d$ such that $\Vert q_i-\epsilon q_j\Vert \geq 3c_1$ for any $i\not =j$ and any $\epsilon\in\{1,-1\}$.
\end{lemma}
\begin{proof}
	Take a maximal set of points satisfying the condition above. Then the balls of radius $3c_1$ around the points $\pm q_1,\dots,\pm q_L$ cover the entire sphere. But any such ball covers a proportion of surface area at most $(Bc_1)^d$ for some constant $B$ (by Lemma \ref{lemma_spherebounds}). Therefore $2L(Bc_1)^d\geq 1$, which gives the result.
\end{proof}

One can prove Theorem \ref{theorem_rainbow} using the method described above. However, the proof naturally yields the more general Theorem \ref{theorem_rbgeneral} (which is restated below), so that is what we shall do. Essentially, we can prove a lower bound of $n^{m-o(1)}$ for a graph $H$ whenever we can draw $H$ in $\RR^m$ in such a way that for each colour there is a line through the origin meeting (the line of) each edge of that colour, and the vertices of the graph span $\RR^m$.

\rainbowgeneral*

\begin{proof}
	Passing to a subspace, we may assume that $m=m_0$ and $\{p_v: v\in V(H)\}$ spans $\RR^m$. Furthermore, by rescaling we may assume that each $z_\kappa$ and each non-zero $p_v$ has unit length. Write $V_0=\{v\in V(H): p_v=0\}$ and $V_1=\{v\in V(H): p_v\not =0\}$. For each $c>0$ and any two positive integers $N,d$, we define a (random) graph $F_{N,d,c}$ as follows. The vertex set of $F_{N,d,c}$ has $|H|$ parts labelled by the vertices of $H$. If $v\in V_0$ then there is a single point $x_{v,1}=0$ in the part labelled by $v$. If $v\in V_1$, then we pick (uniformly and independently at random) $N$ points $x_{v,1},\dots,x_{v,N}$ on $S^d$: these will be the vertices in the part labelled by $v$. We join two vertices $x_{v,a}$ and $x_{w,b}$ by an edge if and only if $vw\in E(H)$ and $|\langle x_{v,a},x_{w,b}\rangle-\langle p_v,p_w\rangle|<c$.

	By assumption, we know that for each edge $vw$ of colour $\kappa$ there exist $\lambda_{\kappa,v}, \lambda_{\kappa,w}$ non-zero real coefficients such that $z_\kappa=\lambda_{\kappa,v}p_v+\lambda_{\kappa,w}p_w$.	
	Let $\lambda$ be the minimum and $M_0$ the maximum over all values of $|\lambda_{\kappa,v}|$. Write $c_1=(12M_0^2c)^{1/2}$. Form a new graph $F_{N,d,c}'$ out of $F_{N,d,c}$ by removing any vertex $x_{v,i}$ for which there is another vertex $x_{v,j}$ (with $j\not =i$) such that $\Vert x_{v,i}-x_{v,j}\Vert \leq \frac{2}{\lambda}c_1$. (The exact values of the constants are not particularly important -- they were chosen so that the graph described below will be properly coloured with no rainbow $K_r$. That is, we could replace $12M_0^2$ and $2/\lambda$ by other sufficiently large constants.)
	
	Let $q_1,\dots, q_L$ be points on $S^d$ with $L\geq (\delta/c_1)^d$ such that $\Vert q_i-q_j\Vert \geq 3c_1$ for all $i\not =j$. Here $\delta$ is some positive (absolute) constant, and the existence of such a set follows from Lemma \ref{lemma_directions}. Also, pick independently and uniformly at random a rotation $R_\kappa\in \so(d+1)$ for each colour $\kappa$ used in the edge-colouring of $H$.
	The probability measure we use on $\so(d+1)$ is the usual (Haar) measure, so for any $q\in S^d$ the points $R_\kappa q$ are independently and uniformly distributed on $S^d$. We think of the points $R_\kappa q_l$ ($l=1,\dots,L$) as the allowed colours for the edges $x_{v,i}x_{w,j}$ when $vw\in E(H)$ has colour $\kappa$ (and we take different rotations for different colours to have independence).
	
	We form an edge-coloured graph $F_{N,d,c}''$ from $F_{N,d,c}'$ as follows. For any edge $x_{v,i}x_{w,j}$ of $F_{N,d,c}'$, we perform the following modification. Let $\kappa$ be the colour of $vw$ in $E(H)$, and let $\lambda_{\kappa,v}, \lambda_{\kappa,w}\not =0$ be as before, so that $z_\kappa=\lambda_{\kappa,v}p_v+\lambda_{\kappa,w}p_w$.
	\begin{itemize}
		\item If there is some $l$ with $\Vert \lambda_{\kappa,v}x_{v,i}+\lambda_{\kappa,w}x_{w,j}-R_\kappa q_l\Vert <c_1$, then we colour the edge $x_{v,i}x_{w,j}$ with colour $(\kappa, l)$. Note that such an $l$ must be unique since $\Vert R_\kappa q_l-R_\kappa q_{l'}\Vert \geq 3c_1$ if $l'\not =l$.
		\item Otherwise we delete the edge $x_{v,i}x_{w,j}$.
	\end{itemize}

\textbf{Claim 1.} The edge-colouring of $F_{N,d,c}''$ is proper.

\textbf{Proof.} Suppose that $x_{v,i}x_{w,j}$ and $x_{v,i}x_{w',j'}$ are both edges with colour $(\kappa, l)$. Then $vw$ and $vw'$ both have colour $\kappa$ in $E(H)$, thus $w=w'$. Also,
\begin{align*}
\Vert x_{w,j'}-x_{w,j}\Vert &\leq \frac{1}{|\lambda_{\kappa,w}|}\left(\Vert \lambda_{\kappa,v}x_{v,i}+\lambda_{\kappa,w}x_{w,j'}-R_\kappa q_l\Vert + \Vert \lambda_{\kappa,v}x_{v,i}+\lambda_{\kappa,w}x_{w,j}-R_\kappa q_l\Vert\right)\\
&\leq \frac{1}{|\lambda_{\kappa,w}|}2c_1\\
&\leq \frac{2}{\lambda}c_1.
\end{align*}
But then $j=j'$ by the definition of $F_{N,d,c}'$. So the edge-colouring of $F_{N,d,c}''$ is indeed proper.\qed\medskip

\textbf{Claim 2.} There is no rainbow copy of $K_r$ in $F_{N,d,c}''$.

\textbf{Proof.} Suppose that the vertices $x_{v_1,i_1},\dots,x_{v_r,i_r}$ form a $K_r$ in $F_{N,d,c}''$. Then $v_1, \dots, v_r$ form a $K_r$ in $H$. This $K_r$ is not rainbow (by assumption). By symmetry, we may assume that the edges $v_1v_2$ and $v_3v_4$ both have colour $\kappa$. Write $x_{a}$ for $x_{v_a,i_a}$ and $\lambda_a$ for $\lambda_{\kappa,v_a}$ for $a=1,2,3,4$. Then we have (recalling that $M_0=\max_{\kappa',v}{|\lambda_{\kappa',v}|}$)
\begin{align*}
\Vert \lambda_1x_1+\lambda_2x_2&-\lambda_3x_3-\lambda_4x_4\Vert^2\\
&=\langle \lambda_1x_1+\lambda_2x_2-\lambda_3x_3-\lambda_4x_4,\lambda_1x_1+\lambda_2x_2-\lambda_3x_3-\lambda_4x_4\rangle\\
&=\sum_{a=1}^{4}\lambda_a^2\Vert x_a\Vert^2+2\lambda_1\lambda_2\langle x_1,x_2 \rangle-2\lambda_1\lambda_3\langle x_1,x_3\rangle-2\lambda_1\lambda_4\langle x_1,x_4\rangle\\
&\hspace{2cm}-2\lambda_2\lambda_3\langle x_2,x_3\rangle-2\lambda_2\lambda_4\langle x_2,x_4\rangle+2\lambda_3\lambda_4\langle x_3,x_4\rangle\\
&\leq \sum_{a=1}^{4}\lambda_a^2\Vert p_{v_a}\Vert^2 +2\lambda_1\lambda_2\langle p_{v_1},p_{v_2} \rangle-2\lambda_1\lambda_3\langle p_{v_1},p_{v_3}\rangle-2\lambda_1\lambda_4\langle p_{v_1},p_{v_4}\rangle\\
&\hspace{2cm}-2\lambda_2\lambda_3\langle p_{v_2},p_{v_3}\rangle-2\lambda_2\lambda_4\langle p_{v_2},p_{v_4}\rangle+2\lambda_3\lambda_4\langle p_{v_3},p_{v_4}\rangle+12M_0^2c\\
&=\langle \lambda_1p_{v_1}+\lambda_2p_{v_2}-\lambda_3p_{v_3}-\lambda_4p_{v_4}, \lambda_1p_{v_1}+\lambda_2p_{v_2}-\lambda_3p_{v_3}-\lambda_4p_{v_4}\rangle+12M_0^2c\\
&=12M_0^2c.
\end{align*}

Since $c_1=(12M_0^2c)^{1/2}$, we get that $\Vert \lambda_1x_1+\lambda_2x_2-\lambda_3x_3-\lambda_4x_4\Vert\leq c_1$. But if $x_1x_2$ has colour $(\kappa,l)$ and $x_3x_4$ has colour $(\kappa,l')$, then 
\[\Vert q_l-q_{l'}\Vert\leq \Vert \lambda_1x_1+\lambda_2x_2-R_\kappa q_l\Vert+\Vert \lambda_3x_3+\lambda_4x_4-R_\kappa q_{l'}\Vert+\Vert \lambda_1x_1+\lambda_2x_2-\lambda_3x_3-\lambda_4x_4\Vert< 3c_1.\] 
It follows that $l=l'$ and hence the $K_r$ with vertices $x_{v_1,i_1},\dots,x_{v_r,i_r}$ is not rainbow.\qed

\textbf{Claim 3.} The expected number of copies of $H$ in $F_{N,d,c}''$ is at least $N^{m-o(1)}$ if $d$ and $c$ are chosen appropriately.

\textbf{Proof.}  Pick arbitrary vertices $x_{v,i_v}$ in the classes (with $i_v=1$ if $v\in V_0$ and $1\leq i_v\leq N$ otherwise). We consider the probability that they form a copy of $H$ in $F_{N,d,c}''$. Write $x_v$ for $x_{v,i_v}$.

Let $\epsilon>0$ be a small constant to be specified later. 
By Lemma \ref{lemma_probabilityKs}, we have
\begin{equation}\label{eq_innerproductsgood}
\PP[|\langle x_v,x_w\rangle-\langle p_v,p_w\rangle|<\epsilon c\textnormal{ for all $v,w\in V(H)$}]\geq \alpha^d(\epsilon c)^{d(|V_1|-m)/2+h}
\end{equation}
for some constants $\alpha>0$ and $h$.

Let $v\in V_1$. By Lemma \ref{lemma_spherebounds}, the probability that $x_v$ is removed when we form $F_{N,d,c}'$ is at most $NB_1^dc^{d/2}$, for some constant $B_1>0$ that does not depend on $N,d,c$. By independence, if $NB_1^dc^{d/2}<1$ then
\begin{equation}\label{eq_notremoved}
\PP[\textnormal{none of the $x_v$ are removed when we form $F_{N,d,c}'$}]\geq (1-NB_1^dc^{d/2})^{|V_1|}.
\end{equation}
Finally, for each colour $\kappa$ in the colouring of $E(H)$, pick an edge $v_\kappa w_\kappa$ of that colour in $H$. Write $y_\kappa=\lambda_{\kappa,v_\kappa}x_{v_\kappa}+\lambda_{\kappa,w_\kappa}x_{w_\kappa}$ and $y_\kappa'=\frac{y_\kappa}{\Vert y_\kappa\Vert}$. Note that $\Vert y_\kappa\Vert\not =0$ with probability 1, since all $\lambda_{\kappa,v}$ are non-zero and at least one of $p_{v_\kappa}$ and $p_{w_{\kappa}}$ is non-zero. For each $\kappa$, if $\epsilon$ is sufficiently small then by Lemma \ref{lemma_spherebounds} we have 
\begin{equation}\label{eq_colourgood}
\PP[\textnormal{there is some $l_\kappa$ such that $\Vert y_\kappa'-R_\kappa q_{l_\kappa}\Vert <\epsilon c^{1/2}$}]\geq L\eta^d(\epsilon c^{1/2})^d\geq \eta_1^d\epsilon^d
\end{equation}
for some constants $\eta,\eta_1>0$.

Observe that the events in \eqref{eq_innerproductsgood}, \eqref{eq_notremoved} and \eqref{eq_colourgood} (for all $\kappa$) are independent. It follows that 
\begin{equation}\label{eq_allhold}
\PP[\textnormal{the events in \eqref{eq_innerproductsgood}, \eqref{eq_notremoved}, and, for all $\kappa$, \eqref{eq_colourgood} hold}]\geq \gamma_\epsilon^dc^{d(|V_1|-m)/2+h}(1-NB_1^dc^{d/2})^{|V_1|}
\end{equation}
where $\gamma_\epsilon$ is some constant depending on $\epsilon$ (but not on $N,d,c$). We show that these events together imply that the $x_v$ form a copy of $H$, if $\epsilon$ is sufficiently small. The only property that we need to check is that no edge is removed when $F_{N,d,c}''$ is formed out of $F_{N,d,c}'$. Consider then an edge $uu'$ of $H$. Let $\kappa$ be its colour and write $v=v_\kappa, w=w_\kappa, y=y_\kappa, y'=y_\kappa', \lambda_v=\lambda_{\kappa,v}, \lambda_w=\lambda_{\kappa,w}, \lambda_u=\lambda_{\kappa,u}$, and $\lambda_{u'}=\lambda_{\kappa,u'}$. We have
\begin{align*}
\langle y,y\rangle
&=\langle \lambda_vx_v+\lambda_wx_w,\lambda_vx_v+\lambda_wx_w\rangle\\
&=\langle \lambda_vp_v+\lambda_wp_w,\lambda_vp_v+\lambda_wp_w\rangle+O(\epsilon c)\\
&=\langle z_\kappa,z_\kappa\rangle+O(\epsilon c)\\
&=1+O(\epsilon c).
\end{align*}
So
\[\Vert y-y'\Vert =|\Vert y\Vert-1|=O(\epsilon c).\]
Furthermore, if we write $y''=\lambda_ux_u+\lambda_{u'}x_{u'}$, then
\begin{align*}
\langle y-y'',y-y''\rangle
&=\langle \lambda_vx_v+\lambda_wx_w-\lambda_ux_u-\lambda_{u'}x_{u'}, \lambda_vx_v+\lambda_wx_w-\lambda_ux_u-\lambda_{u'}x_{u'} \rangle\\
&=\langle \lambda_vp_v+\lambda_wp_w-\lambda_up_u-\lambda_{u'}p_{u'}, \lambda_vp_v+\lambda_wp_w-\lambda_up_u-\lambda_{u'}p_{u'} \rangle+O(\epsilon c)\\
&=\langle z_\kappa-z_\kappa,z_\kappa-z_\kappa\rangle +O(\epsilon c)\\
&=O(\epsilon c).
\end{align*}
It follows that
\begin{align*}
\Vert y''-R_\kappa q_{l_\kappa}\Vert
&\leq \Vert y''-y\Vert +\Vert y-y'\Vert +\Vert y'-R_\kappa q_{l_\kappa}\Vert\\
&\leq O((\epsilon c)^{1/2})+O(\epsilon c)+\epsilon c^{1/2}.
\end{align*}
This is indeed less than $c_1=(12M_0^2c)^{1/2}$ if $\epsilon$ is sufficiently small.\medskip

Choosing $\epsilon$ appropriately, $\eqref{eq_allhold}$ gives that the expected number of copies of $H$ in $F_{N,d,c}''$ is at least
\[N^{|V_1|}\gamma^dc^{(|V_1|-m)d/2+h}(1-NB_1^dc^{d/2})^{|V_1|}\]
for some constant $\gamma$. Letting $c=\left(\frac{1}{2NB_1^d}\right)^{2/d}$ and $d=\lfloor\sqrt{\log{N}}\rfloor$, we get that the expected number of copies of $H$ in $F_{N,d,c}''$ is at least $N^{m-o(1)}$, which proves the claim.\qed

The theorem follows from Claims 1, 2 and 3.
\end{proof}

\begin{proof}[Deduction of Theorem \ref{theorem_rainbow} from Theorem \ref{theorem_rbgeneral}]
Given a complete graph $K_r$ on vertex set $\{1,\dots,r\}$, we can properly edge-colour it by giving the edges $12$ and $34$ the same colour $\kappa$, and giving arbitrary different colours to the remaining edges. Pick $r-1$ linearly independent points $p_2, p_3, \dots, p_r$ in $\RR^{r-1}$, and let $p_1=p_2+p_3+p_4$. Let $z_\kappa=p_3+p_4=p_1-p_2$ and let $z_{\kappa'}=p_i+p_j$ when $ij$ is an edge of colour $\kappa'\not =\kappa$. Theorem \ref{theorem_rbgeneral} gives that $\ex(n,K_r,\textnormal{rainbow-}K_r)\geq n^{r-1-o(1)}$, and we have a matching upper bound by Proposition \ref{proposition_rbHH}.
\end{proof}

\section{Some applications of Theorem \ref{theorem_rbgeneral}}\label{sec_examples}

We have already seen that Theorem \ref{theorem_rbgeneral} can be used to answer the question of Gerbner, Mészáros, Methuku and Palmer about the order of magnitude of $\ex(n,K_r,\textnormal{rainbow-}K_r)$. In this section we give some other examples of applications of the theorem. 

To show that our lower bounds are sharp, we shall use a simple proposition to give matching upper bounds. This will require the following definition. Given a graph $H$ and a proper edge-colouring $c$ of $H$, we say that a subset $V_0\subseteq V(H)$ is a \textit{$c$-spanning set} if there is an ordering $v_1,\dots,v_k$ of the vertices in $V(H)\setminus V_0$ such that for all $i$ there are some $u,u',w\in V_0\cup\{v_1,\dots,v_{i-1}\}$ such that $uu'\in E(H)$, $v_iw\in E(H)$ and $c(uu')=c(v_iw)$. In other words, we can add the remaining vertices to $V_0$ one by one in a way that new vertices are joined to some vertex in the set by a colour already used.

\begin{proposition}\label{proposition_rbsharp}
Let $H$ and $F$ be graphs and let $r$ be a positive integer. Assume that for every proper edge-colouring $c$ of $H$ that does not contain a rainbow copy of $F$ there is a $c$-spanning set of size at most $r$. Then $\ex(n,H,\textnormal{rainbow-}F)=O(n^r)$. If we also have $r<|V(H)|$, and if for every such $c$ and every edge $e$ of $H$ there is a $c$-spanning set of size at most $r$ containing $e$, then $\ex(n,H,\textnormal{rainbow-}F)=o(n^r)$.
\end{proposition}
\begin{proof}
	Let $G$ be a graph on $n$ vertices and let $\kappa$ be a proper edge-colouring of $G$ without a rainbow copy of $F$. Let $G$ contain $M$ copies of $H$. Then we can partition the vertices into classes $X_v$ for $v\in V(H)$ in such a way that there are $\Omega(M)$ choices of $\mathbf{x}=(x_v)_{v\in V(H)}$ such that $x_v\in X_v$ and $v\mapsto x_v$ is a graph homomorphism from $H$. (To see this, place each vertex independently, uniformly at random into one of the classes. If $\{x_v: v\in V(H)\}$ is an isomorphic copy of $H$ in $G$ (such that $v\mapsto x_v$ is the corresponding isomorphism), then we have $\PP[x_v\in X_v\textnormal{ for all $v$}]=1/|V(H)|^{|V(H)|}$, so the expected number of such tuples $\mathbf{x}$ is $M/|V(H)|^{|V(H)|}=\Omega(M)$.)
	
	For each $\mathbf{x}$ as above pick an isomorphic proper edge-colouring $c_{\mathbf{x}}: E(H)\to \{1,\dots, |E(H)|\}$, that is, $c_{\mathbf{x}}(vw)=c_{\mathbf{x}}(v'w')$ if and only if $\kappa(x_vx_w)=\kappa(x_{v'}x_{w'})$ for all edges $vw,v'w'$ of $H$. Note that $c_\mathbf{x}$ cannot contain a rainbow copy of $F$. Then there is a colouring $c: E(H)\to \{1,\dots, |E(H)|\}$ that appears for $\Omega(M)$ choices of $\mathbf{x}$. Let $V_0$ be a $c$-spanning set of size at most $r$.
	
	Note that any $\mathbf{x}$ with $c_\mathbf{x}=c$ is determined by $(x_v)_{v\in V_0}$, since the edge-colouring is proper. But there are $O(n^r)$ choices for $(x_v)_{v\in V_0}$, hence $M=O(n^r)$.
	
	Now assume that $r<|V(H)|$ and that for every proper edge-colouring $c'$ of $H$ without a rainbow $F$ and every edge $e$ of $H$ there is a $c'$-spanning set of size at most $r$ that contains $e$. By the graph removal lemma and the first part of our proposition, we can remove $o(n^2)$ edges from $G$ so that the new graph $G'$ contains no copy of $H$. So it suffices to show that each edge appeared in at most $O(n^{r-2})$ tuples $\mathbf{x}$ with $c_\mathbf{x}=c$. Given an edge $e=y_vy_w$ with $y_v\in X_v, y_w\in X_w, vw\in E(H)$ we can pick in $H$ a $c$-spanning set $V_{0,e}$ of size at most $r$ containing $vw$. Then any $\mathbf{x}$ with $c_\textbf{x}=c$ and $x_v=y_v, x_w=y_w$ is determined by $(x_u)_{u\in V_0\setminus\{v,w\}}$, which gives the result.
\end{proof}

Now we give some sample applications of Theorem \ref{theorem_rbgeneral} and Proposition \ref{proposition_rbsharp}. We shall give two illustrations, but it is quite easy to generate additional examples.

\subsection{Complete graphs}\label{rbkrks}
Perhaps the most natural extension of Question \ref{question_rbKr} is to determine the behaviour of the function $\ex(n,K_r,\textnormal{rainbow-}K_s)$. Note that trivially $\ex(n,K_r,\textnormal{rainbow-}K_s)=\Theta(n^r)$ when $s>r$ (by taking a complete $r$-partite graph), and we have seen that $\ex(n,K_s,\textnormal{rainbow-}K_s)=n^{s-1-o(1)}$ (when $s\geq 4$). We also have $\ex(n,K_r,\textnormal{rainbow-}K_s)=0$ whenever $r\geq r_s$ for some integer $r_s$ depending on $s$. Indeed, if we have a $K_r$ with no rainbow copy of $K_s$, and the largest rainbow subgraph has order $t\leq s$, then any of the remaining $(r-t)$ vertices must be joined to this $K_t$ by one of the $\binom{t}{2}$ colours appearing in the $K_t$. But each such colour appears at most once at each vertex, giving $r=O(s^3)$. {In fact, Alon, Lefmann and Rödl showed \cite{alon1991anti} that $r_s=\Theta(s^3/\log{s})$.}

However, the question is non-trivial for $s< r< r_s$.
First note that $\ex(n,K_r,\textnormal{rainbow-}K_s)=o(n^{s-1})$ whenever $r\geq s$ by Proposition \ref{proposition_rbsharp} (since any maximal rainbow subgraph is a $c$-spanning set). The simplest case for the lower bound is $(r,s)=(5,4)$. In this case Theorem \ref{theorem_rbgeneral} gives a matching lower bound $n^{3-o(1)}$. Indeed, take an arbitrary proper edge-colouring of $K_5$ with no rainbow $K_4$, and take points $p_1,\dots,p_5$ in general position in $\RR^3$. The existence of appropriate values of $z_\kappa$ follows from the fact that any four of the $p_i$ are linearly dependent (but any three are independent), and each colour is used at most twice. It is easy to deduce that $\ex(n,K_{s+1},\textnormal{rainbow-}K_s)=n^{s-1-o(1)}$ for all $s\geq 4$.

When $s=4$ then $r_s=7$ (since any triangle is in at most one $K_4$), leaving the case $(r,s)=(6,4)$. 
Unfortunately, in this case Theorem \ref{theorem_rbgeneral} does not give a lower bound of $n^{3-o(1)}$. (To see this, observe that to get such a bound the corresponding points $p_v$ would all have to be non-zero. Then we can use the alternative formulation Theorem \ref{theorem_rbgeneral}$'$ to see that we would have to be able to draw a properly edge-coloured $K_6$ in the plane such that there is no rainbow $K_4$ and lines of edges of the same colour are either all parallel or go through the same point. Applying an appropriate projection and affine transformation, we may assume that we have two colour classes where the edges are all parallel, and these two parallel directions are perpendicular. This leaves essentially two cases to be checked, and neither of them yields an appropriate configuration.)

However, we can still deduce a lower bound of $\ex(n,K_6,\textnormal{rainbow-}K_4)\geq n^{12/5-o(1)}$, as sketched below. We can take 6 points $p_0=0$ and $p_a=e^{2\pi i a/5}$ (for $a=1,\dots,5$), that is, the vertices of a regular pentagon together with its centre. We define a colouring $c$ as follows. Give parallel lines between vertices of the pentagon the same colour, and also give the same colour to the edge incident at the centre which is perpendicular to these lines (see Figure \ref{K6K4}). This gives a proper edge-colouring of $K_6$ and corresponding points in 2 dimensions for which the conditions of Theorem \ref{theorem_rbgeneral} are satisfied, giving a lower bound of $n^{2-o(1)}$. (The point $z_\kappa$ is chosen to be $p_a$ when $p_0p_a$ has colour $\kappa$.) This can be improved to $n^{12/5-o(1)}$ by a product argument as follows. Looking at the construction, we see that our graph $G$ is $6$-partite with classes $V_0,\dots,V_5$, at most $n$ vertices, and a proper edge-colouring $\kappa$ such that the following hold.
\begin{itemize}
	\item There are (at least) $n^{2-o(1)}$ copies of $K_6$ in $G$.
	\item The class $|V_0|$ has size $1$.
	\item There is a $5$-colouring $c$ of the edges of $K_6$ (on vertex set $\{0,\dots,5\}$) with no rainbow $K_4$ such that whenever $v_{i_1},v_{i_2},v_{i_3},v_{i_4}$ form a $K_4$ in $G$ with $v_{i_j}\in V_{i_j}$, then $i_j\mapsto v_{i_j}$ gives an isomorphism of colourings between the restrictions of $c$ and $\kappa$ to the appropriate four-vertex graphs (i.e., $\kappa(v_{i_j}v_{i_l})=\kappa(v_{i_{j'}}v_{i_{l'}})$ if and only if $c(i_ji_{j'})=c(i_{j'}i_{l'})$). Moreover, this 5-colouring $c$ has the property that for all $i,j\in \{1,\dots,6\}$ there is a permutation of the vertices $\{0,\dots,5\}$ which is an automorphism of colourings and maps $i$ to $j$. (Indeed, we can take rotations of the pentagon when $i,j\not =0$, and we can take the permutation $(01)(34)$ when $i=0$, $j=1$.)

\end{itemize} 

We construct a new graph as follows. For each $i\in\{0,\dots,5\}$, pick a permutation $\pi_i$ of $\{0,\dots,5\}$ which gives a colouring automorphism of $c$ and sends $i$ to $0$. Define a $6$-partite graph $G_i$ obtained from $G$ by permuting the vertex classes: $G_i$ has classes $V_0^i,\dots,V_5^i$ given by $V_a^i=V_{\pi_i(a)}$ and same edge set as $G$. Let $G'$ be the product of these $6$-partite graphs, that is, it is $6$-partite with vertex classes $W_a=V_a^0\times V_a^1\times\dots\times V_a^5$, and two vertices $(v_0,\dots,v_5)\in W_a$ and  $(w_0,\dots,w_5)\in W_b$ are joined by an edge if $v_iw_i\in E(G)$ for all $i$. Moreover, colour such an edge by colour $(\kappa(v_0w_0),\dots,\kappa(v_5w_5))$. It is easy to check that the colouring is proper, $G'$ contains no rainbow $K_4$, $G'$ has at most $n^5$ vertices in each class, and $G'$ contains at least $n^{12-o(1)}$ copies of $K_6$, giving the bound stated.

\begin{figure}[h]
	\includegraphics[clip,trim=0.5cm 0.4cm 1cm 0.1cm,
	width=0.4\linewidth]{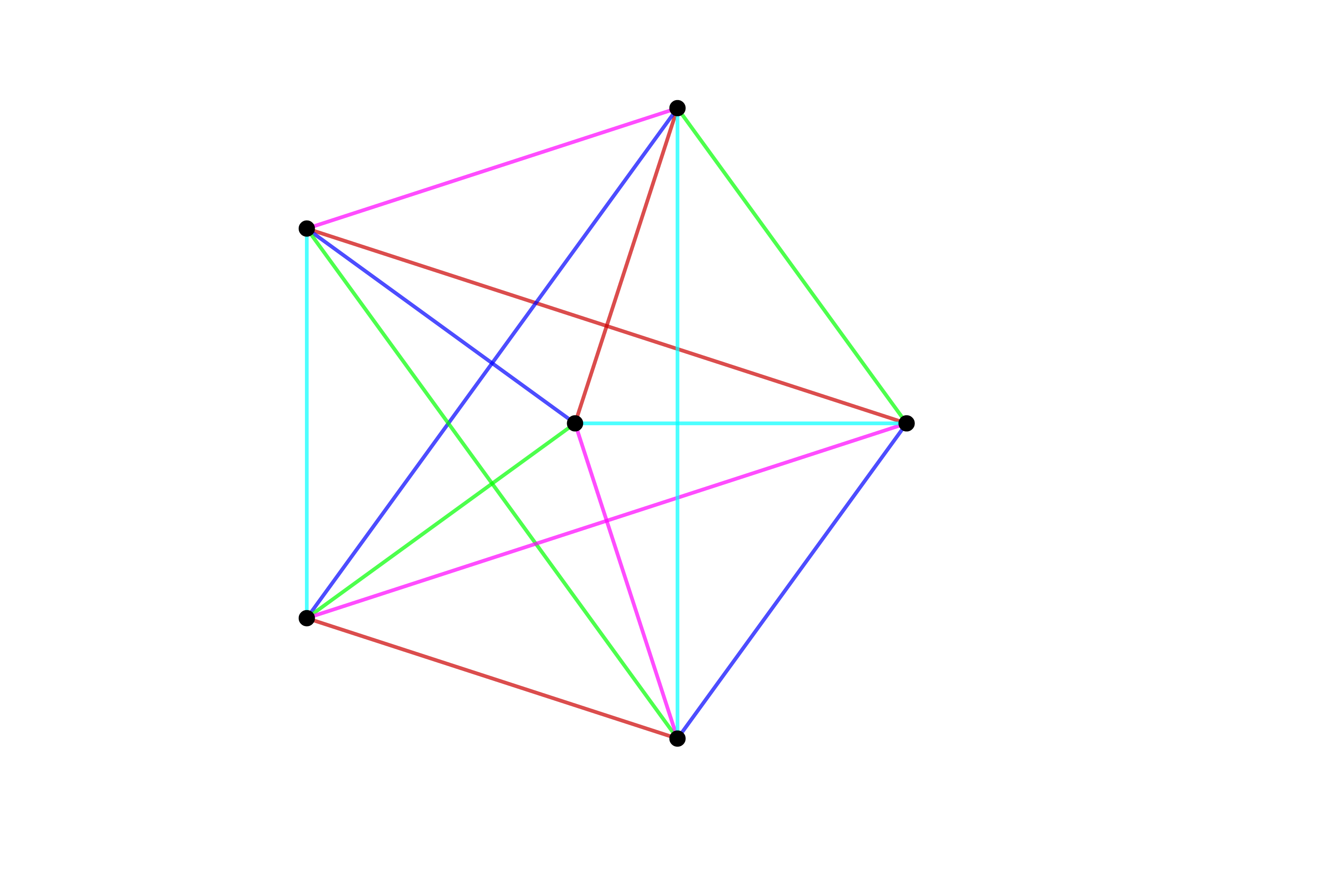}
	\centering
	\caption{The colouring and points used for $(r,s)=(6,4)$ to get a lower bound.}
	\label{K6K4}
\end{figure}

This leaves some open questions about $\ex(n,K_r,\textnormal{rainbow-}K_s)$. It would be interesting to determine its order of magnitude for $(r,s)=(6,4)$, or the magnitude for other pairs with $s<r<r_s$.

\subsection{King's graphs}

Given positive integers $k,l\geq 2$, write $H_{k,l}$ for the graph with vertex set $\{1,\dots,k\}\times\{1,\dots,l\}$ where $(a,b)$ and $(a',b')$ are joined by an edge if and only if they are distinct and $|a-a'|,|b-b'|\leq 1$. In other words, $H_{k,l}$ is the strong product of a path with $k$ points and a path with $l$ points, sometimes called the $k\times l$ king's graph. We can use our results to show that $\ex(n,H_{k,l},\textnormal{rainbow-}K_4)=n^{k+l-1-o(1)}$.

First consider the upper bound. It is easy to see that any sequence of vertices $p_1,\dots,p_{k+l-1}$ is a $c$-spanning set (for all proper edge-colourings $c$ of $H_{k,l}$ without a rainbow $K_4$) if either of the following statements holds.
\begin{enumerate}
	\item We have $p_1=(1,1)$, $p_{k+l-1}=(k,l)$ and $p_{i+1}-p_i\in\{(0,1),(1,0)\}$ for all $i$.
	\item We have $p_1=(1,l)$, $p_{k+l-1}=(k,1)$ and $p_{i+1}-p_i\in\{(0,1),(-1,0)\}$ for all $i$.
\end{enumerate}
(Indeed, this follows from the fact that we can add the other vertices one by one, creating a new copy of $K_4$ in our set in each step.) Since any edge is contained in such a sequence, Proposition \ref{proposition_rbsharp} gives $\ex(n,H_{k,l},\textnormal{rainbow-}K_4)=o(n^{k+l-1})$.

For the lower bound, consider an edge-colouring $c$ of $H_{k,l}$ with $c((a,b)(a+1,b))=a$, where the other edges are given arbitrary distinct colours. This gives a proper edge-colouring of $H_{k,l}$ with no rainbow $K_4$. For each vertex $(a,b)$ of $H$, define $p_{a,b}\in \RR^{k+l}$ to be the vector with $i$\thth coordinate 
\begin{equation*}
(p_{a,b})_i =
\begin{cases*}
0 & if 
$i\not =a, k+b$ \\
1 & if $i=a$\\
(-1)^a & if $i=k+b$
\end{cases*}
\end{equation*}
For each $1\leq a\leq k-1$ we let $z_a\in \RR^{k+l}$ be the vector with all entries zero except the $a$\thth and $(a+1)$\thth coordinates which are $1$, and for each other colour $\kappa$ used in the colouring of $H_{k,l}$ we take $z_\kappa=p_v+p_w$, where $vw$ is the unique edge of colour $\kappa$. Then we have $p_{a,b}+p_{a+1,b}=z_a$, so the conditions of Theorem \ref{theorem_rbgeneral} are satisfied. The dimension of the subspace of $\RR^{k+l}$ spanned by the vectors $p_{a,b}$ is at least $k+l-1$, since $p_{1,l},p_{1,l-1},\dots,p_{1,1},p_{2,1},p_{3,1},\dots,p_{k,1}$ are linearly independent. We get the required lower bound $n^{k+l-1-o(1)}$.

\section*{Acknowledgement} We are grateful to Shagnik Das for pointing out that the behaviour of $r_s$ (in Subsection \ref{rbkrks}) was known.

\bibliography{Bibliography}
\bibliographystyle{abbrv}	

\appendix
\section{Appendix}
In this appendix, we prove Lemma \ref{lemma_probabilityKs}, which is recalled below.
\appendixlemma*
	
\begin{lemma}\label{lemma_r+1points}
	Let $r$ be a positive integer. Let $p_1,\dots,p_{r+1}$ be points on $S^{r-1}$ such that $p_1,\dots, p_r$ are linearly independent. Then there exist real numbers $\delta>0$, $\alpha_0>0$ and $h_0$ such that whenever $d\geq r$ is a positive integer, $0<c<1$, and $x_1,\dots,x_r$ are points on $S^d$ with $|\langle x_i,x_j\rangle -\langle p_i, p_j\rangle |<\delta c$ for all $1\leq i,j\leq r$, then the probability that a random point $x_{r+1}$ on $S^d$ satisfies $|\langle x_i,x_{r+1}\rangle -\langle p_i, p_{r+1}\rangle|<c$ for all $i\leq r$ is at least $\alpha_0^dc^{d/2+h_0}$.
\end{lemma}
\begin{proof}
	We prove the statement by induction on $r$. If $r=1$, then $p_1,p_2\in\{-1,1\}$ and the condition $|\langle x_i,x_{r+1}\rangle -\langle p_i, p_{r+1}\rangle|<c$ becomes $|\langle p_1p_2x_1,x_2\rangle-1|<c$, which is equivalent to $\Vert p_1p_2x_1-x_2\Vert <\sqrt{2c}$. By Lemma \ref{lemma_spherebounds}, this happens with probability at least $\alpha_0^dc^{d/2}$, giving the claim. (Here $\delta=1$ and $h_0=0$.)
	
	Now assume that $r\geq 2$ and the result holds for smaller values of $r$. We may assume that $p_{r+1}\not =\pm p_1$ (otherwise swap $p_1$ and $p_2$). By symmetry, we may assume that $x_1=(0,0,\dots,0,1)\in S^d$ and $p_1=(0,0,\dots,0,1)\in S^{r-1}$. Write $x_a^{i}$ for the $i$\thth coordinate of $x_a$. For each $2\leq a\leq r+1$, define a normalized projected vector
	\[x_a'=\frac{(x_a^{1},x_a^{2},\dots,x_a^{d})}{\Vert (x_a^{1},x_a^{2},\dots,x_a^{d})\Vert }\in S^{d-1}.\]
	Note that the denominator is non-zero for $a=2,\dots,r$ if $\delta$ is sufficiently small, and it is non-zero with probability 1 for $a=r+1$. Also, $x_{r+1}$ is uniformly distributed on $S^{d-1}$. Similarly, for each $2\leq a\leq r+1$, define
	\[p_a'=\frac{(p_a^{1},p_a^{2},\dots,p_a^{r-1})}{\Vert (p_a^{1},p_a^{2},\dots,p_a^{r-1})\Vert }\in S^{r-2}.\]
	 Note that $p_2', \dots, p_r'$ are linearly independent in $\RR^{r-1}$. 
	
	Note that for $2\leq a,b\leq r$ we have
	\[\langle x_a',x_b'\rangle=\frac{\langle x_a,x_b\rangle-\langle x_1,x_a\rangle\langle x_1,x_b\rangle}{(1-\langle x_1,x_a\rangle^2)^{1/2}(1-\langle x_1,x_b\rangle^2)^{1/2}}= \frac{\langle p_a,p_b\rangle-\langle p_1,p_a\rangle\langle p_1,p_b\rangle}{(1-\langle p_1,p_a\rangle^2)^{1/2}(1-\langle p_1,p_b\rangle^2)^{1/2}}+O(\delta c)=\langle p_a',p_b'\rangle+O(\delta c).\]
	Let $\epsilon>0$ be a small constant to be specified later.
	Applying the induction hypothesis for $r'=r-1$ and points $p_2',\dots,p_{r+1}'$,
	we have that
	\begin{equation}\label{eq_probcomplete1}
	\PP[|\langle x_a',x_{r+1}'\rangle-\langle p_a',p_{r+1}'\rangle|<\epsilon c\textnormal{ for all $2\leq a\leq r$}]\geq \alpha_0^{d-1}(\epsilon c)^{(d-1)/2+h_0}
	\end{equation}
	whenever $\delta<\delta_0\epsilon$, for some constants $\alpha_0, \delta_0>0$ and $h_0$ depending on $p_1,\dots,p_{r+1}$ only.
	
	By Lemma \ref{lemma_spherebounds}, there is a constant $\beta$ depending on $p_1,p_{r+1}$ only such that
	\begin{equation}\label{eq_probinY1}
	\PP[|\langle x_1,x_{r+1}\rangle-\langle p_1,p_{r+1}\rangle|<\epsilon c]\geq \beta^d\epsilon c.
	\end{equation}
	
	Note that the events in \eqref{eq_probcomplete1} and \eqref{eq_probinY1} are independent, since conditioning on the second event we still have an independent uniform distribution for the vector $x_{r+1}'$. It follows that 
	\begin{equation}\label{eq_probboth1}
	\PP[|\langle x_a',x_{r+1}'\rangle-\langle p_a',p_{r+1}'\rangle|<\epsilon c\textnormal{ for all $2\leq a\leq r$ and } |\langle x_1,x_{r+1}\rangle-\langle p_1,p_{r+1}\rangle|<\epsilon c]\geq \gamma^d(\epsilon c)^{d/2+h_1}
	\end{equation}
	whenever $\delta<\delta_0\epsilon$, for some constants $\gamma>0$ 
	and $h_1$ (with the constants depending on $p_1,\dots,p_{r+1}$ only). 
	
	So it suffices to show that if $\epsilon$ and $\delta$ are sufficiently small (depending on $p_1,\dots,p_{r+1}$ only), then the event above implies that $|\langle x_a,x_{r+1}\rangle-\langle p_a,p_{r+1}\rangle|<c$ for all $2\leq a\leq r$. But we have
	\begin{align*}\langle x_a,x_{r+1}\rangle&=\langle x_a',x_{r+1}'\rangle (1-\langle x_1,x_{a}\rangle^2)^{1/2}(1-\langle x_1,x_{r+1}\rangle^2)^{1/2} +\langle x_1,x_{a}\rangle \langle x_1,x_{r+1}\rangle\\
	&=(\langle p_a',p_{r+1}'\rangle+O(\epsilon c)) ((1-\langle p_1,p_{a}\rangle^2)^{1/2}+O(\delta c))((1-\langle p_1,p_{r+1}\rangle^2)^{1/2}+O(\epsilon c))\\
	&\hspace{0.5cm} +(\langle p_1,p_{a}\rangle+O(\delta c)) (\langle p_1,p_{r+1}\rangle+O(\epsilon c))\\
	&=\langle p_a,p_{r+1}\rangle +O((\epsilon+\delta)c),
	\end{align*}
	which gives the result.
\end{proof}

\begin{lemma}\label{lemma_indeppoints}
	Let $r$ be a positive integer and let $p_1,\dots,p_r$ be linearly independent points in $S^{r-1}$. Then there are exist real numbers $\alpha_1>0$ and $h_1$ such that whenever $d\geq r$ is a positive integer and $0<c<1$ then the probability that $r$ points $x_1,\dots,x_r$ chosen independently and uniformly at random on $S^d$ satisfy $|\langle x_i,x_j\rangle -\langle p_i, p_j\rangle |< c$ for all $1\leq i,j\leq r$ is at least $\alpha_1^dc^{h_1}$.
\end{lemma}

\begin{proof}
	The proof is essentially the same as for the previous lemma. We prove the statement by induction on $r$. The case $r=1$ is trivial. Now assume that $r\geq 2$ and that the statement holds for smaller values of $r$.  By symmetry, we may assume that $x_1=(0,0,\dots,0,1)\in S^d$ and $p_1=(0,0,\dots,0,1)\in S^{r-1}$. Define $p_a'$ and $x_a'$ for $a\geq 2$ as in the proof of Lemma \ref{lemma_r+1points}. Let $\epsilon>0$ be some small constant to be determined later.
	
	By induction, we have
	\begin{equation*}\label{eq_probcomplete2}
	\PP[|\langle x_a',x_{b}'\rangle-\langle p_a',p_{b}'\rangle|<\epsilon c\textnormal{ for all $2\leq a,b\leq r$}]\geq \alpha_1^{d-1}(\epsilon c)^{h_1}
	\end{equation*}
	for some $\alpha_1>0$ and $h_1$ (where the constants depend only on $p_1,\dots,p_{r+1}$).
	
	By Lemma \ref{lemma_spherebounds}, there are constants $\beta_2,\dots,\beta_r$ 
	depending on $p_1,\dots,p_{r}$ only such that for each $2\leq a\leq r$
	\begin{equation*}\label{eq_probinY2}
	\PP[|\langle x_1,x_{a}\rangle-\langle p_1,p_{a}\rangle|<\epsilon c]\geq \beta_a^d\epsilon c.
	\end{equation*}
	
	By independence,
	\[	\PP[|\langle x_a',x_{b}'\rangle-\langle p_a',p_b'\rangle|<\epsilon c\textnormal{ for all $2\leq a,b\leq r$ and } |\langle x_1,x_{a}\rangle-\langle p_1,p_{a}\rangle|<\epsilon c\textnormal{ for all $a$}]\geq \gamma^d(\epsilon c)^{h_2}\]
	for some real numbers $\gamma>0$ 
	and $h_2$.
	
	However, if the event above holds then
	\begin{align*}\langle x_a,x_b\rangle&=\langle x_a',x_b'\rangle (1-\langle x_1,x_{a}\rangle^2)^{1/2}(1-\langle x_1,x_b\rangle^2)^{1/2} +\langle x_1,x_{a}\rangle \langle x_1,x_b\rangle\\
	&=\langle p_a',p_b'\rangle (1-\langle p_1,p_{a}\rangle^2)^{1/2}(1-\langle p_1,p_b\rangle^2)^{1/2} +\langle p_1,p_{a}\rangle \langle p_1,p_b\rangle+O(\epsilon c)\\
	&=\langle p_a,p_b\rangle +O(\epsilon c).
	\end{align*}
	The result follows by taking a sufficiently small $\epsilon$.
\end{proof}

\begin{proof}[Proof of Lemma \ref{lemma_probabilityKs}]
	By Lemma \ref{lemma_r+1points}, we can choose constants $0<\delta<1$, $\alpha_0>0$ and $h_0$ such that whenever $d\geq r$ is a positive integer, $0<c<1$ and $x_1,\dots,x_r$ are points on $S^d$ with $|\langle x_i,x_j\rangle -\langle p_i, p_j\rangle |<\delta c$ for all $1\leq i,j\leq r$, then for all $a>r$ the probability that a random point $x_a$ on $S^d$ satisfies $|\langle x_i,x_{a}\rangle -\langle p_i, p_{a}\rangle|<c$ for all $i\leq r$ is at least $\alpha_0^dc^{d/2+h_0}$.
	
	Now let $\epsilon$ be a small constant to be specified later. Using Lemma \ref{lemma_indeppoints}, the observation above and independence of $x_{r+1},\dots,x_s$ conditional on $x_1,\dots,x_r$, we have that
	\begin{align*}\PP[ |\langle x_i,x_j\rangle-\langle p_i,p_j\rangle|< \delta\epsilon c\textnormal{ whenever $i,j\leq r$ and }&|\langle x_i,x_a\rangle-\langle p_i,p_a\rangle|< \epsilon c\textnormal{ whenever $i\leq r<a\leq s$}]\\
	&\geq \alpha_1^d(\delta\epsilon c)^{h_1}\alpha_0^{(s-r)d}(\epsilon c)^{(s-r)(d/2+h_0)}\\
	&\geq \alpha^d(\epsilon c)^{d(s-r)/2+h}
	\end{align*}
	for some constants $\alpha>0$ and $h$. We show that the event above implies that $|\langle x_a,x_b\rangle-\langle p_a,p_b\rangle|<c$ even if $a,b>r$ (if $\epsilon$ is sufficiently small).
	Given $b>r$, we can find coefficients $\lambda_{b,a}$ such that $p_b=\sum_{a=1}^{r}\lambda_{b,a}p_a$. Write $y_b=\sum_{a=1}^{r}\lambda_{b,a}x_a$. Then
	\begin{align*}
	\Vert y_b-x_b\Vert^2&=\left\langle \sum_{a=1}^{r}\lambda_{b,a}x_a-x_b,\sum_{a=1}^{r}\lambda_{b,a}x_a-x_b\right\rangle\\
	&=\left\langle \sum_{a=1}^{r}\lambda_{b,a}p_a-p_b,\sum_{a=1}^{r}\lambda_{b,a}p_a-p_b\right\rangle+O(\epsilon c)\\
	&=O(\epsilon c).
	\end{align*}
	Thus $\Vert y_b-x_b\Vert=O\left((\epsilon c)^{1/2}\right)$. Furthermore, we have, for each $1\leq i\leq r$,
	\begin{align*}
	\langle x_i,y_b-x_b\rangle&= \left\langle x_i,\sum_{a=1}^{r}\lambda_{b,a}x_a-x_b\right\rangle\\
	&=\left\langle p_i,\sum_{a=1}^{r}\lambda_{b,a}p_a-p_b\right\rangle+O(\epsilon c)\\
	&=O(\epsilon c).
	\end{align*}
	It follows that whenever $b,b'>r$ then
	\begin{align*}
	\langle x_b,x_{b'}\rangle &= \langle y_b+(x_b-y_b),y_{b'}+(x_{b'}-y_{b'})\rangle\\
	&=\langle y_b, y_{b'}\rangle +\langle x_b-y_b,y_{b'}\rangle +\langle y_b,x_{b'}-y_{b'}\rangle +O(\epsilon c)\\
	&=\left\langle\sum_{a=1}^{r}\lambda_{b,a}x_a,\sum_{a=1}^{r}\lambda_{b',a}x_a \right\rangle+
	\left\langle x_b-y_b,\sum_{a=1}^{r}\lambda_{b',a}x_a \right\rangle+
	\left\langle\sum_{a=1}^{r}\lambda_{b,a}x_a,x_{b'}-y_{b'} \right\rangle+O(\epsilon c)\\
	&=\left\langle\sum_{a=1}^{r}\lambda_{b,a}p_a,\sum_{a=1}^{r}\lambda_{b',a}p_a \right\rangle+O(\epsilon c)\\
	&=\langle p_b,p_{b'}\rangle+O(\epsilon c).
	\end{align*}
	Choosing a sufficiently small $\epsilon>0$ gives the result.
\end{proof}

\end{document}